\documentclass[11pt]{amsart}
\usepackage{amssymb,mathrsfs}
\usepackage{color,umoline}
\usepackage[dvipsnames]{xcolor}
\usepackage{graphicx}
\usepackage{enumitem}
\usepackage[font=small,labelfont=bf]{caption}
\usepackage{tikz, caption, subcaption}
\usepackage{relsize}
\usepackage[mathscr]{eucal}
\usepackage{verbatim}
\usepackage[draft]{todonotes}
\usepackage{hyperref}
\usepackage[dvipsnames]{xcolor}
\usetikzlibrary{patterns}
\tikzstyle{EDR}=[draw=lightgray,line width=0pt,preaction={clip, postaction={pattern=north east lines, pattern color=gray}}]
\tikzstyle{EDR1}=[draw=lightgray,line width=0pt,preaction={clip, postaction={pattern=north west lines, pattern color=gray}}]

\definecolor{mygray}{gray}{0.95}

\definecolor{mypink1}{rgb}{1.2,1.1,0.9}

\definecolor{mypink2}{rgb}{1.0,0.95 ,0.9}

\definecolor{mypink3}{rgb}{1.0,0.6,0.7}

\numberwithin{equation}{section}

\newtheorem{theorem}{Theorem}[section]

\newtheorem{lemma}[theorem]{Lemma}

\newtheorem{proposition}[theorem]{Proposition}
\newtheorem{remark}[theorem]{Remark}

\numberwithin{equation}{section}

\newcommand{\beq}{\begin{equation}}
	\newcommand{\eeq}{\end{equation}}
\newcommand{\beqq}{\begin{equation*}}
	\newcommand{\eeqq}{\end{equation*}}
\newcommand{\ben}{\begin{eqnarray}}
	\newcommand{\een}{\end{eqnarray}}
\newcommand{\beno}{\begin{eqnarray*}}
	\newcommand{\eeno}{\end{eqnarray*}}

\begin{document}
\setcounter{page}{1}
\title[Fourier restriction estimates on quantum Euclidean spaces]
{Fourier restriction estimates on quantum Euclidean spaces }

\author[]{Guixiang Hong}
\address{
School of Mathematics and Statistics\\
Wuhan University\\
Wuhan 430072\\
China}

\email{guixiang.hong@whu.edu.cn}

\author[]{Xudong Lai}
\address{
Institute for Advanced Study in Mathematics\\
Harbin Institute of Technology\\
Harbin
150001\\
China}

\email{xudonglai@hit.edu.cn}
\author[]{Liang Wang}
\address{
School of Mathematics and Statistics\\
Wuhan University\\
Wuhan 430072\\
China}

\email{wlmath@whu.edu.cn}

\thanks{}


\date{\today
}
\begin{abstract}
In this paper, we initiate the study of the Fourier restriction phenomena on quantum Euclidean spaces, and establish the analogues of the Tomas-Stein restriction theorem and the two-dimensional full restriction theorem.
\end{abstract}

\maketitle
\section{Introduction}

Quantum Euclidean spaces and quantum tori which are the model examples of noncommutative locally compact manifolds, have appeared frequently in the literature of mathematical physics, such as string theory and noncommutative field theory \cite{DN01, NS98, SW99}. In recent years, harmonic analysis on this  noncommutative manifold  has been developed very rapidly, such as the pointwise convergence of Fourier series and functional spaces on quantum tori \cite{CXY, Lai21, Xiong18}, singular integral theory and  pseudodifferential operator
theory \cite{GJM, GJP, JPMX}, the commutator estimates and quantum differentiability for quantum Euclidean spaces \cite{ LSZ, Xiong1}.

In brief, quantum Euclidean space $\mathcal{R}_\theta^d$ is the von Neumann subalgebra of $\mathcal{B}(L_2(\mathbb{R}^d))$ generated by $\{U_\theta(t)\}_{t\in\mathbb{R}^d}$, where $\{U_\theta(t)\}_{t\in\mathbb{R}^d}$ is a family of unitary operators on $L_2(\mathbb{R}^d)$ satisfying the following Weyl relation:
\beq U_\theta(t)U_\theta(s)=e^{\frac i2(s,\theta t)}U_\theta(t+s),\quad \text{for\: all}\: t,s\in\mathbb{R}^d,\eeq
where $\theta$ is a $d\times d $ real  antisymmetric matrix. The algebra $\mathcal{R}_\theta^d$ admits a normal semifinite faithful trace $\tau_\theta$, and $L_p(\mathcal{R}_\theta^d)$ is the noncommutative $L_p$ space associated to $(\mathcal{R}_\theta^d, \tau_\theta)$. Note that if $\theta=0$, $L_p(\mathcal{R}_\theta^d)$ is the usual $L_p$ space defined on $\mathbb{R}^d$ with the Lebesgue measure. We refer the reader to next section for more details.

The aim of the present paper is to investigate the Fourier restriction phenomena in this setting.
 For $x\in \mathcal{S}(\mathcal{R}_\theta^d)$, the Schwartz class on $\mathcal{R}_\theta^d$, we define the Fourier transform of $x$ as
\begin{equation*}
  \hat{x}(\xi):=\tau_\theta(xU_\theta(\xi)^*),\quad \text{for}\:\xi\in\mathbb{R}^d.
\end{equation*}
As in the commutative case ($\theta=0$), the Plancherel theorem and the Hausdorff-Young inequalities remain true in $\mathcal{R}_\theta^d$, which are enough to guarantee the fact that the Fourier transform $x\rightarrow\hat{x}$ extends to a continuous  linear map from $L_p(\mathcal{R}_\theta^d)$ to $L_{p^\prime}(\mathbb{R}^d)$ for $1\leq p\leq2,\frac{1}{p}+\frac{1}{p^\prime}=1$ (see Lemma \ref{hy*} for details).
In other words, when $1\leq p\leq2$, given $x\in L_p(\mathcal{R}_\theta^d)$, we have $\hat{x}\in L_{p^\prime}(\mathbb{R}^d)$, which is not defined pointwise but only defined in the sense of almost everywhere. The restriction problem asks: for $x\in L_p(\mathcal{R}_\theta^d)$, whether there is a meaningful restriction of $\hat{x}$ to a subset $\Sigma \subset\mathbb{R}^d$ with Lebesgue measure zero?

When $x\in L_1(\mathcal{R}_\theta^d)$, $\hat{x}$ coincides with a function that is continuous on $\mathbb{R}^d$ and tends to zero at infinity which follows from the Riemann-Lebesgue lemma (see Lemma \ref{RL}), thus $\hat{x}$ can be meaningfully restricted to any set $\Sigma$. At the other extreme, for an arbitrary $x\in L_2(\mathcal{R}_\theta^d)$, by  the Plancherel theorem, $\hat{x}$ is also an arbitrary function in $L_2(\mathbb{R}^d)$, then there is no meaningful restriction to any zero measure set $\Sigma$. Between these two extremes, it is a natural question that what happens to the Fourier transform of $x\in L_p(\mathcal{R}_\theta^d)$ when $1<p<2$. In the case $\theta=0$, it was observed by Stein \cite{Fe1} in 1960s that  for some  measure zero subset $\Sigma\subset\mathbb{R}^d$, such as the unit sphere $S^{d-1}$, it might be possible to obtain meaningful restrictions of the Fourier transforms of functions $f\in L_p(\mathbb{R}^d)$ for some $1\leq p<2$.
Stein formulated the quantitative version of the Fourier restriction problem as follows: for which $1\leq p, q\leq\infty$, the following estimate holds for all Schwartz functions $f$ on $\mathbb{R}^d$,
\begin{equation}\label{resconjecture}
 \|\hat{f}\|_{L_q(S^{d-1}, \;d\sigma)}\leq C_{d,p,q}\|f\|_{L_p(\mathbb{R}^d)}.
\end{equation}
 Here $d\sigma$ is the surface measure on the unit sphere $S^{d-1}$. The Knapp example shows that  when the estimate \eqref{resconjecture} holds we must have  $1\leq p<\frac{2d}{d+1}$ and $q\leq \frac{d-1}{d+1}p^\prime$ (see e.g. Page 291 of \cite{MS}). And Stein's restriction conjecture is that the necessary conditions claimed above are also sufficient. That is, it is conjectured that the estimate \eqref{resconjecture} holds for all $p,q$ satisfying $1\leq p<\frac{2d}{d+1}$ and $q\leq \frac{d-1}{d+1}p^\prime$.
Stein and Fefferman \cite{Fe1}, Zygmund \cite{Zy1} solved this conjecture for $d=2$. When $d\geq3$, the full conjecture is still open except the special case $q=2$ which was proved by Tomas and Stein in \cite{To1, To2}. That is, the estimate \eqref{resconjecture} holds for all $1\leq p\leq 2(d+1)/(d+3)$ when $q=2$.

   In this paper, we first establish an analogue of the full restriction theorem on $\mathcal{R}_\theta^2$ for all $2\times 2 $ real  antisymmetric matrix $\theta$:
\begin{theorem}\label{rt1}
Let $\theta$ be a $2\times 2 $ real  antisymmetric matrix and $x\in\mathcal{S}(\mathcal{R}_\theta^2)$. When $1\leq p< \frac{4}{3}$ and $ q\leq {p^\prime}/3$, we have the following estimate
    \beq\label{res1} \|\hat{x}\|_{L_q(S^1)}\leq C_{p,q}\|x\|_{L_p(\mathcal{R}_\theta^2)}, \eeq
   where $S^1$ denotes the unit circle and $C_{p,q}$ is a constant that depends only on $p,q$.
\end{theorem}
The unit circle $S^1$ can be regarded as the limit of the annuli $A_\delta=\{s\in\mathbb{R}^2:\quad 1-\delta<|s|<1+\delta\}$ with $\delta\rightarrow0$. 
For the annulus $A_\delta$, we have the following theorem:
\begin{theorem} \label{res}
 Let $1\leq p< 4/3,q=p^\prime/3$ and $0<\delta<1/2$. Then for all  $x\in\mathcal{S}(\mathcal{R}_\theta^2)$, we have
 \beq\label{rrr}
 \|\chi^\delta \hat{x}\|_{L_q(\mathbb{R}^2)}\leq C_{p,q}\delta ^{1/q}\|x\|_{L_p(\mathcal{R}_\theta^2)},\eeq
 where \beqq \chi^\delta(\xi):=\chi_{(1-\delta,1+\delta)}(|\xi|) .\eeqq
 Moreover, when $p=4/3$
 \beq\label{rrr1}
  \|\chi^\delta \hat{x}\|_{L_{4/3}(\mathbb{R}^2)}\leq C\delta ^{3/4}(\log{\delta^{-1}})^{1/4}\|x\|_{L_{4/3}(\mathcal{R}_\theta^2)}.\eeq
\end{theorem}

To establish \eqref{rrr} and \eqref{rrr1}, we come across several difficulties when adapting the classical arguments---Theorem 5.4.7 of \cite{GF}. Roughly speaking, the fact that $|x^*x|^2$ is not necessarily equal to $|xx|^2$ for a general operator $x$ yields some new inequalities and geometric structures that need to be handled. We refer to Remark \ref{diff} for the descriptions of two of them. Fortunately, based on the noncommutative analysis technologies developed in recent years and some new geometric observations,  we overcome these difficulties to provide a proof of the endpoint estimate \eqref{rrr1}. On the other hand, unlike the classical case where the arguments for the endpoint estimate work directly for the non-endpoint ones, we are still unable to show \eqref{rrr} using the proof of \eqref{rrr1} (see Remark \ref{diff} (ii)).
However as in the classical case $\theta=0$, one may check that \eqref{res1} is equivalent to \eqref{rrr} using the change of variable arguments (see the beginning of Section \ref{s4}), and the latter is not trivial on quantum Euclidean spaces (see e.g. Proposition \ref{trans}). There are many ways to establish \eqref{res1} in the classical case. At the moment of writing, we find that the only approach that works in the noncommutative setting is due to Zygmund \cite{Zy1}, see Section \ref{s3} for the adaptation.

When $d\geq3$, we get the  restriction estimate of Tomas-Stein exponents as below:
\begin{theorem}\label{rt2}
Let $d\geq3$ and $\theta$ be a $d\times d $ real  antisymmetric matrix. When $1\leq p\leq\frac{2(d+1)}{d+3}$,  for all $x\in\mathcal{S}(\mathcal{R}_\theta^d)$, we have the following estimate
    \beq\label{res2} \|\hat{x}\|_{L_2(S^{d-1},d\sigma)}\leq C_{d,p}\|x\|_{L_p(\mathcal{R}_\theta^d)}, \eeq
 where $C_{d,p}$ is a constant that depends only on $d,p$.
\end{theorem}

A standard method (see e.g. Theorem 5.4.5 of \cite{GF}) to prove Theorem \ref{rt2} in classical Euclidean space $\mathbb{R}^d$ is to split the $d$-dimensional spatial variable into two lower dimensional variables. It should be pointed out that this method is not available for general quantum Euclidean spaces. However, we find that the method of analytic interpolation developed by Stein and Tomas \cite{To2} is feasible, see Section \ref{s5} for more details.

In the case $\theta=0$, the Fourier restriction conjecture is one of the most important problems and is enormously influential in harmonic analysis, geometric measure theory and partial differential equations. In fact, the study of the Fourier restriction conjecture is quite related to the Bochner-Riesz conjecture, the Kakeya conjecture, local smoothing problems, Strichartz estimates and Schr\"odinger maximal inequalities etc. (see e.g. \cite{De, Fe2, HZ,Ro, Str, Tao}). In the subsequent works, we will see how to apply Theorem \ref{rt1} and \ref{rt2} and their proof to deal with some problems related to the local smoothing phenomenon and Schr\"odinger maximal inequalities on quantum Euclidean spaces. Moreover, in classical harmonic analysis, after Bourgain's  pioneering work \cite{bo} in 1990s, which first broke up the Tomas-Stein exponents, numerous great progresses have been made around the restriction problem in recent years (see e.g. \cite{BD, Guth, HZ,Tao03,Wang}). It is natural to ask whether one can extend these results to the noncommutative case.
These extensions seem quite challenging, since all the methods to improve the Tomas-Stein exponents are highly dependent on the geometric structures of $\mathbb{R}^d$, which are unavailable on the noncommutative manifold $\mathcal{R}_\theta^d$.



The rest of this paper is organized as follows. In Section \ref{s2}, we present definitions, notions and notation mentioned above. In Sections \ref{s3} and \ref{s4}, we prove Theorems \ref{rt1} and \ref{res} respectively. In Section \ref{s5}, based on the Young-type inequalities on $\mathcal{R}_\theta^d$, we combine $TT^*$ method and Stein's analytic interpolation theorem to prove Theorem \ref{rt2}.

 \textbf{Notation:}  In what follows, we write
$A\lesssim_\alpha B$ if $A\le C_\alpha B$ for some constant $C_\alpha>0$ only depending on the index $\alpha$, and we write $A\thickapprox B$ to mean that $A\lesssim B$ and $B \lesssim A$.
For a function $f$, we set $\tilde{f}(\cdot):=f(-\cdot)$. Let $\mathbb{H}$ be a Hilbert space, denote the inner product by $\langle\cdot, \cdot\rangle_\mathbb{H}$, which is linear in the second variable and conjugate linear in the first variable.

\section{Preliminaries and some lemmas}\label{s2}
\subsection{Noncommutative $L_p$ spaces}
Let $\mathcal{M}$ be a semifinite von Neumann algebra equipped with a normal semifinite faithful trace ($n.s.f.$ in short) $\tau$. The support of $x\in\mathcal{M}$ is the least projection such that $ex=x=xe$, denoted by $s(x)$. Let $\mathcal{S}_+(\mathcal{M})$ be the set of all $x\in\mathcal{M}_+$ such that $\tau(s(x))<\infty$, and $\mathcal{S}(\mathcal{M})$ be the linear span of $\mathcal{S}_+(\mathcal{M})$. For $1\leq p<\infty$, $x\in\mathcal{S}(\mathcal{M})$, we define \beqq \|x\|_p=(\tau(|x|^p))^{1/p},\eeqq
where $|x|=(x^*x)^{1/2}$ is the modulus of $x$. The quantity $\|\cdot\|_p$ is a norm, and thus $(\mathcal{S}(\mathcal{M}),\|\cdot\|_p)$ forms a normed vector space. We then denote the completion of $(\mathcal{S}(\mathcal{M}),\|\cdot\|_p)$ by $L_p(\mathcal{M})$, which is called the noncommutative $L_p$ space associated to $(\mathcal{M},\tau)$. For $p=\infty$, we set $L_\infty(\mathcal{M})=\mathcal{M}$, and $\|x\|_\infty:=\|x\|_\mathcal{M}$. Like classical $L_p$ spaces, the noncommutative $L_p$ spaces enjoy the basic properties such as the duality, the interpolation etc.. For more information on noncommutative $L_p$ spaces, we refer to \cite{PX03}.
\subsection{Quantum Euclidean spaces }
 Let $\theta$ be a $d\times d $ real  antisymmetric matrix and $t\in\mathbb{R}^d$, we define the unitary
operator on $L_2(\mathbb{R}^d)$:
\beq\label{wll}
		(U_\theta(t)f)(r):=e^{-\frac i2(t,\theta r)}f(r-t),\quad f\in{L_2(\mathbb{R}^d)}, r\in\mathbb{R}^d.
	\eeq
It is easy to verify that the family $\{U_\theta(t)\}_{t\in\mathbb{R}^d}$ is strongly continuous. For $r,s,t\in\mathbb{R}^d, f\in{L_2(\mathbb{R}^d)}$, we have
\begin{align*}
(U_\theta(t)U_\theta(s)f)(r)&=U_\theta(t)(e^{-\frac i2(s,\theta r)}f(r-s))\\
 & =e^{-\frac i2(t,\theta r)}e^{-\frac i2(s,\theta (r-t))}f(r-s-t)\\
 &=(e^{\frac i2(s,\theta t)}U_\theta(t+s)f)(r).
\end{align*}
Hence we have the Weyl relation $U_\theta(t)U_\theta(s)=e^{\frac i2(s,\theta t)}U_\theta(t+s)$ and $U_{\theta}^*(t)=U_\theta(-t)$.
We call the von Neumann subalgebra of $\mathcal{B}(L_2(\mathbb{R}^d))$ generated by $\{U_\theta(t)\}_{t\in\mathbb{R}^d}$ quantum Euclidean spaces, which is denoted   by $\mathcal{R}_\theta^d $. We refer the reader to e.g. \cite{GJP, LSZ, Xiong1} for more information on  $\mathcal{R}_\theta^d $.
\begin{remark}{\rm
  In the case $\theta=0$, $\mathcal{R}_0^d $ is the von Neumann algebra generated by the unitary group of translations on $\mathbb{R}^d$, which is
  $*$-isomorphic to $L_\infty(\mathbb{R}^d)$.}
\end{remark}

 We will introduce a map from $L_1(\mathbb{R}^d)$ to $\mathcal{R}_\theta^d $ which is still denoted by $U_\theta$  as below:
Let $f \in L_1(\mathbb{R}^d)$, one defines $U_\theta(f): L_2(\mathbb{R}^d)\rightarrow L_2(\mathbb{R}^d)$ as
\begin{equation}\label{defU}
 U_\theta(f)( g):=\int_{\mathbb{R}^d}f(t)(U_\theta(t)g)dt
\end{equation}
for $g\in L_2(\mathbb{R}^d)$.
This $L_2(\mathbb{R}^d)$-valued integral is convergent in the Bochner sense. As in \cite{Xiong1}, one can see that $U_\theta(f)\in \mathcal{R}_\theta^d$ and $U_\theta$ is injective.
The image of $\mathcal{S}(\mathbb{R}^d)$, the Schwartz class on $\mathbb R^d$, under the map $U_\theta$ is called the class of Schwartz functions on $\mathcal{R}_\theta^d $ :
$$ \mathcal{S}(\mathcal{R}_\theta^d):=\{x\in\mathcal{R}_\theta^d:\quad x=U_\theta(f), \text{\:for\: some\;} f\in\mathcal{S}(\mathbb{R}^d)\}.$$
Then $U_\theta$ is a bijection from $\mathcal{S}(\mathbb{R}^d)$ to $ \mathcal{S}(\mathcal{R}_\theta^d)$,
and thus $\mathcal{S}(\mathcal{R}_\theta^d)$ is a Fr\'{e}chet topological space equipped with the Fr\'{e}chet topology induced by $U_\theta$. We denote the space of continuous linear functionals on $\mathcal{S}(\mathcal{R}_\theta^d)$ as $\mathcal{S}^\prime(\mathcal{R}_\theta^d)$, then $U_\theta$ extends to a bijection from $\mathcal{S}^\prime(\mathbb{R}^d)$  to $\mathcal{S}^\prime(\mathcal{R}_\theta^d)$: for $f\in\mathcal{S}^\prime(\mathbb{R}^d)$,
\beq (U_\theta(f),U_\theta(g)):=(f, \tilde{g}), \quad \text{for\:all\;}g\in \mathcal{S}(\mathbb{R}^d).\eeq
If $x\in \mathcal{S}(\mathcal{R}_\theta^d)$ is given by $x=U_\theta(f)$ for $f\in \mathcal{S}(\mathbb{R}^d)$, we define $\tau_\theta(x):=f(0)$, then $\tau_\theta$ extends to a $n.s.f.$ trace on $\mathcal{R}_\theta^d$. The associated noncommutative $L_p$ space is denoted by $L_p(\mathcal{R}_\theta^d)$. The space $ \mathcal{S}(\mathcal{R}_\theta^d)$ is dense in $L_p(\mathcal{R}_\theta^d)$ for $1\leq p<\infty$ with respect to the norm, and dense in $L_\infty(\mathcal{R}_\theta^d)$ in the weak* topology. See \cite{GJP, Xiong1} for more information.

\begin{remark}
{\rm
(i). For $t\in\mathbb R^d$, $U_\theta(t)$ plays the role of $\exp_t(\cdot)=e^{2\pi i\langle t,\cdot\rangle}$ on $\mathbb R^d$. Then for $f\in\mathcal{S}(\mathbb{R}^d)$, $U_\theta(f)$ can be viewed as the inverse Fourier transform of $f$ on $\mathbb R^d$ denoted by $\check{f}$ and $\tau_\theta(U_\theta(f))$ can be regarded as the Lebesgue integral of $\check{f}$.

(ii). For $x\in L_1(\mathcal{R}_\theta^d)+L_\infty(\mathcal{R}_\theta^d)$, $x$ can be embedded into $\mathcal{S}^\prime(\mathcal{R}_\theta^d)$ by
 \beqq (x,y)=\tau_\theta(xy) \quad \text{for\: all\;} y\in \mathcal{S}(\mathcal{R}_\theta^d).\eeqq

}
\end{remark}
\begin{lemma}\label{f*g}
    For $f,g\in\mathcal{S}(\mathbb{R}^d)$, one has
    \beqq U_\theta(f)^*=U_\theta(\overline{\tilde{f}})\eeqq and\beqq U_\theta(f)U_\theta(g)=U_\theta(f*_\theta g),\eeqq where
    $f*_\theta g(s)=\int_{\mathbb{R}^d}f(t)g(s-t)e^{\frac{i}{2}(s,\theta t)}dt$.
  \end{lemma}
  \begin{proof}
  Note that $U_\theta(f)^*=\int_{\mathbb{R}^d}\overline{f(t)}U_\theta(t)^*dt=\int_{\mathbb{R}^d}\overline{f(t)}U_\theta(-t)dt=\int_{\mathbb{R}^d}\overline{f(-t)}U_\theta(t)dt$,
  and
    \begin{align*}
      U_\theta(f)U_\theta(g) & = \int_{\mathbb{R}^d}f(t)U_\theta(t)dt\int_{\mathbb{R}^d}g(s)U_\theta(s)ds \\
    &=\int_{\mathbb{R}^d}\int_{\mathbb{R}^d}f(t)g(s)e^{\frac{i}{2}(s,\theta t)}U_\theta(s+t)dtds\\
    &=\int_{\mathbb{R}^d}\left(\int_{\mathbb{R}^d}f(t)g(s-t)e^{\frac{i}{2}(s,\theta t)}dt\right)U_\theta(s)ds.
    \end{align*}
    Thus we get the desired equalities. 
  \end{proof}

 The Fourier transform on quantum Euclidean space is defined as $\hat{x}(\xi):=\tau_\theta(xU_\theta(\xi)^*)$ for $x\in \mathcal{S}(\mathcal{R}_\theta^d)$. It may be viewed as the inverse map of $U_\theta$. Indeed, if $x=U_\theta(f)\in \mathcal{S}(\mathcal{R}_\theta^d)$ for some $f\in\mathcal{S}(\mathbb{R}^d)$, as argued in Lemma \ref{f*g}, we have $\hat{x}=f$ pointwise.
 Based on Lemma \ref{f*g}, one can get the analogue of the Plancherel theorem on quantum Euclidean spaces. Moreover, we also have the Hausdorff-Young inequalities. 
 \begin{lemma}\label{hy*}
    Let $x\in\mathcal{S}(\mathcal{R}_\theta^d)$, we have
\beq \label{hy3}
\|\hat{x}\|_{L_2(\mathbb{R}^d)}=\|x\|_{L_{2}(\mathcal{R}_\theta^d)};\eeq
and  for $1\leq p<2$, \beq \label{hy4}
\|\hat{x}\|_{L_{p^\prime}(\mathbb{R}^d)}\leq\|x\|_{L_{p}(\mathcal{R}_\theta^d)}.\eeq
Thus the Fourier transform $\hat{}$ can extend to a contraction from $L_p(\mathcal{R}_\theta^d)$ ($1\leq p\leq2$) to $L_{p^\prime}(\mathbb{R}^d)$ and moreover an isometry on $L_2(\mathcal{R}_\theta^d)$.
  \end{lemma}
 \begin{proof}
 Suppose that $x=U_\theta(f)$ for $f\in \mathcal{S}(\mathbb{R}^d)$, by Lemma \ref{f*g},
  \begin{align*}
     \|U_\theta(f)\|_{L_{2}(\mathcal{R}_\theta^d)}^2&=\tau_\theta(U_\theta(f)^*U_\theta(f))\\&=\tau_\theta(U_\theta(\overline{\tilde{f}}*_\theta f))\\
   &=(\overline{\tilde{f}}*_\theta f)(0)\\&=\|f\|_{L_2(\mathbb{R}^d)}^2=\|\hat{x}\|_{L_2(\mathbb{R}^d)}.
  \end{align*}
The result for $p=1$ can be deduced easily from the definition of the Fourier transform and the noncommutative H\"{o}lder inequality. Then the desired inequalities for other $p$'s follow from the standard complex interpolation arguments (see e.g. \cite{PX03}).
  Since $\mathcal{S}(\mathcal{R}_\theta^d)$ is dense in $L_p(\mathcal{R}_\theta^d)$, $\hat{}$ extends to a contraction from $L_p(\mathcal{R}_\theta^d)$ to $L_{p^\prime}(\mathbb{R}^d)$.
  \end{proof}

  Moreover, in the case $p=1$, one has the following Riemann-Lebesgue type lemma.
\begin{lemma}\label{RL}
Let $x\in L_1(\mathcal{R}_\theta^d)$, then $\hat{x}$ coincides with a continuous function on $\mathbb{R}^d$ which tends to zero at infinity.
\end{lemma}
\begin{proof}
  For $x\in L_1(\mathcal{R}_\theta^d)$, by the density of $\mathcal{S}(\mathcal{R}_\theta^d)$ in $L_1(\mathcal{R}_\theta^d)$, we can find a Cauchy sequence $\{x_n\}_{n=1}^\infty\subset \mathcal{S}(\mathcal{R}_\theta^d)$ that converges to $x$ in $L_1(\mathcal{R}_\theta^d)$. By \eqref{hy4}, one may define
   \beqq f(\xi):=\lim_{n\rightarrow\infty}\hat{x}_n(\xi)=\lim_{n\rightarrow\infty}\tau_\theta(x_nU_\theta(\xi)^*).\eeqq   Moreover it is easy to show $\hat{x}_n\rightarrow f$ uniformly since $\{x_n\}_{n=1}^\infty$ is a Cauchy sequence in $L_1(\mathcal{R}_\theta^d)$. Thus $\hat{x}_n\rightarrow f$ in $L_\infty(\mathbb{R}^d)$, and $f$ coincides with $\hat{x}$ which has been defined in
  Lemma \ref{hy*}. Note that $\{\hat{x}_n\}_{n=1}^\infty$ is a sequence of Schwartz functions on $\mathbb{R}^d$ and $\hat{x}_n\rightarrow f$ uniformly, then $f$ is a continuous function on $\mathbb{R}^d$ which tends to zero at infinity.
\end{proof}
\begin{remark}{\rm
 In the case $\theta=0$, $\hat{}$  coincides with the Fourier transform on $\mathbb{R}^d$. In the subsequent sections, we use the notation
   $\hat{}$ to indicate the Fourier transform on $\mathcal{R}_\theta^d$ for all $\theta$ without causing any confusion.
 }
  \end{remark}
  We also have the similar results as in Lemma \ref{hy*} for the inverse Fourier transform $U_\theta$ once we note that for $f\in \mathcal{S}(\mathbb{R}^d)$,
  \begin{equation*}
   \|U_\theta(f)\|_{L_{\infty}(\mathcal{R}_\theta^d)}\leq\|f\|_{L_{1}(\mathbb{R}^d)},
  \end{equation*}
which can be deduced easily from \eqref{defU} and the triangle inequality.
\begin{lemma}\label{HY}
Let $f\in \mathcal{S}(\mathbb{R}^d)$, then we have
\beq \label{hy1}
\|U_\theta(f)\|_{L_{2}(\mathcal{R}_\theta^d)}=\|f\|_{L_2(\mathbb{R}^d)};\eeq
and  for $1\leq p<2$, \beq \label{hy2}
\|U_\theta(f)\|_{L_{p^\prime}(\mathcal{R}_\theta^d)}\leq\|f\|_{L_{p}(\mathbb{R}^d)}.\eeq
Thus $U_\theta$ can extend to a contraction from $L_p(\mathbb{R}^d)$ ($1\leq p\leq2$) to $L_{p^\prime}(\mathcal{R}_\theta^d)$ and moreover an isometry on $L_2(\mathbb{R}^d)$.
\end{lemma}
The proof is similar to that of Lemma \ref{hy*}, we omit the details (see e.g. \cite{Xiong1}).
\section{The full restriction theorem on $\mathcal{R}_\theta^2$}\label{s3}
 Before giving the proof of Theorem \ref{rt1}, we will show that Theorem \ref{rt1} can be reduced to proving \eqref{res1} for an open region $V\subset S^1$ instead of $S^1$. For that purpose, we need the change of variable arguments.
  Let $T$ be an invertible linear transform from $\mathbb{R}^2$ to $\mathbb{R}^2$, then $T$ can be regarded as an invertible real $2\times 2$ matrix.
  We define a map $\Psi_T$ from $L_\infty(\mathcal{R}_\theta^2)$ to $L_\infty(\mathcal{R}_{\theta_T}^2)$ as
  \beqq\Psi_T(U_\theta(s)):=U_{\theta_T}(T^{-1}s),\quad \text{for\: all\;} s\in\mathbb{R}^2.\eeqq
  Here $\theta_T:=T^t\theta T$, $T^t$ is the transpose of $T$. One can check that $\Psi_T$ is a $*$-isomorphism (see e.g. \cite{Xiong1}).
 And we have the following proposition.
  \begin{proposition}\label{trans}
    Let $T$ be an invertible $2\times 2$ real matrix. 
    For $1\leq p\leq\infty$ and $x\in\mathcal{S}(\mathcal{R}_\theta^2)$, we have
    $$\|\Psi_T(x)\|_{L_p(\mathcal{R}_{\theta_T}^2)}=|\det T|^{1/p}\|x\|_{L_p(\mathcal{R}_{\theta}^2)}.$$
  \end{proposition}
  \begin{proof}
    It suffices to show \beq\label{pp1}\|\Psi_T(x)\|_{L_p(\mathcal{R}_{\theta_T}^2)}\leq|\det T|^{1/p}\|x\|_{L_p(\mathcal{R}_{\theta}^2)}\eeq
    for all $x\in\mathcal{S}(\mathcal{R}_\theta^2)$, since we can get the reverse inequality via the map $\Psi_{T^{-1}}$. When $p=2$, for a given  $x\in\mathcal{S}(\mathcal{R}_\theta^2)$, we can find
    $f\in\mathcal{S}(\mathbb{R}^2)$ such that $x=U_\theta(f)$. By Lemma \ref{HY},
    \begin{align*}
     \|\Psi_T(x)\|_{L_2(\mathcal{R}_{\theta_T}^2)}&=|\det T|\|U_{\theta_T}(f(T\cdot))\|_{L_2(\mathcal{R}_{\theta_T}^2)}\\&=|\det T|\|f(T\cdot)\|_{L_2(\mathbb{R}^2)}
      =|\det T|^\frac{1}{2}\|x\|_{L_2(\mathcal{R}_{\theta}^2)}.
    \end{align*}
    When $p=\infty$, $\Psi_T$ is a $*$-isomorphism from $L_\infty(\mathcal{R}_{\theta}^2)$ to $L_\infty(\mathcal{R}_{\theta_T}^2)$, thus $\Psi_T$ is an isometry. Then we get \eqref{pp1} for $2\leq p\leq \infty$ via interpolation.

      When $1\leq p\leq2$,  for a given $x\in\mathcal{S}(\mathcal{R}_\theta^2)$, by duality and the H\"{o}lder inequality,
      \begin{align*}
        \|\Psi_T(x)\|_{L_p(\mathcal{R}_{\theta_T}^2)}&=\sup_{y\in\mathcal{S}(\mathcal{R}_{\theta_T}^2),\|y\|_{p^\prime}=1}|\tau_{\theta_T}(\Psi_T(x)y)|\\&=\sup_{y\in\mathcal{S}(\mathcal{R}_{\theta_T}^2),\|y\|_{p^\prime}=1}
        |\det T||\tau_\theta(x\Psi_{T^{-1}}(y))|\\
        &\leq\sup_{y\in\mathcal{S}(\mathcal{R}_{\theta_T}^2),\|y\|_{p^\prime}=1}|\det T|\|x\|_{L_p(\mathcal{R}_{\theta}^2)}\|\Psi_{T^{-1}}(y)\|_{L_{p^\prime}(\mathcal{R}_{\theta}^2)}\\
        &\leq |\det T|^{1/p}\|x\|_{L_p(\mathcal{R}_{\theta}^2)}.
      \end{align*}
      Here we have used the fact that $\tau_{\theta_T}(\Psi_T(x)(y))=|\det T|\tau_{\theta}(x\Psi_{T^{-1}}(y)).$
      Indeed, suppose that $x=U_\theta(f),y=U_{\theta_T}(g)$ for some $f,g\in\mathcal{S}(\mathbb{R}^2)$. As argued in Lemma \ref{f*g}, we have
    \beqq\tau_{\theta_T}(\Psi_T(x)y)=|\det T|\int_{\mathbb{R}^2}f(Ts)g(-s)ds=\int_{\mathbb{R}^2}f(s)g(-T^{-1}s)ds\eeqq
     and \beqq|\det T|\tau_{\theta}(x\Psi_{T^{-1}}(y))=\int_{\mathbb{R}^2}f(s)g(-T^{-1}s)ds.\eeqq
    \end{proof}
  \begin{proposition}\label{local}
   Let $V$ be an open region of $S^{1}$ and $1\leq p,q\leq\infty$. If the following estimate \beq\label{ress2} \|\hat{x}\|_{L_q(V,d\sigma)}\lesssim_{p,q}\|x\|_{L_p(\mathcal{R}_\theta^2)} \eeq
    holds true for all $2\times 2 $ real  antisymmetric matrix $\theta$ and  all $x\in\mathcal{S}(\mathcal{R}_\theta^2)$, then we have \beq\label{ress1} \|\hat{x}\|_{L_q(S^{1},d\sigma)}\lesssim_{p,q}\|x\|_{L_p(\mathcal{R}_\theta^2)} \eeq
    for all $2\times 2 $ real  antisymmetric matrix $\theta$ and  all $x\in\mathcal{S}(\mathcal{R}_\theta^2)$.
  \end{proposition}
  \begin{proof}
   Since $S^{1}$ is compact, we can cover $S^{1}$ with finite open regions $\{T_i(V)\}_{i=1}^n$, where $T_i(V)$ is the rotation of $V$ via a suitable rotation transform $T_i$. Then
    \beq\label{ress3}\|\hat{x}\|_{L_q(S^{1}, d\sigma)}\leq\sum_{i=1}^n\|\hat{x}\|_{L_q(T_i(V),d\sigma)}.\eeq
    Note that the surface measure $d\sigma$ is invariant under rotation and $|\det T_i|=1$, we have $\|\hat{x}\|_{L_q(T_i(V),d\sigma)}=\|\hat{x}(T_i\cdot)\|_{L_q(V,d\sigma)}$. Then for a given $2\times 2 $ real  antisymmetric matrix $\theta$,
    by \eqref{ress2} and Proposition \ref{trans},
    \begin{align*}
      \|\hat{x}\|_{L_q(T_i(V),d\sigma)}&=\|\hat{x}(T_i\cdot)\|_{L_q(V,d\sigma)}\lesssim_{p,q}\|U_{\theta_{T_i}}(\hat{x}(T_i\cdot))\|_{L_p(\mathcal{R}_{\theta_{T_i}}^2)}\\
    &=\|\Psi_{T_i^{-1}}(U_{\theta_{T_i}}(\hat{x}(T_i\cdot)))\|_{L_p(\mathcal{R}_{\theta}^2)}=\|x\|_{L_p(\mathcal{R}_{\theta}^2)},
    \end{align*}
   where in the final equality we have used the fact that $$\Psi_{T_i^{-1}}(U_{\theta_{T_i}}(\hat{x}(T_i\cdot)))=|\det T_i|^{-1}U_{\theta}(\hat{x}(\cdot))=x.$$
  Finally, the desired estimate  \eqref{ress1} follows from \eqref{ress3} .
  \end{proof}
  \begin{remark}{\rm
  It is easy to see that Proposition \ref{trans} and Proposition \ref{local} are also valid in higher dimensions.}
  \end{remark}
  Now we focus on the proof of Theorem \ref{rt1}.
\begin{proof}[Proof of Theorem \ref{rt1}] When $1\leq p< \frac{4}{3}$ and $ q= {p^\prime}/3$, by Proposition \ref{local}, it suffices to show
that for all $x\in\mathcal{S}(\mathcal{R}_\theta^2)$,
    \beq\label{rres} \|\hat{x}\|_{L_q(S_+^1)}\lesssim_{p}\|x\|_{L_p(\mathcal{R}_\theta^2)}, \eeq
   where $S_+^1=\{(t,\sqrt{1-t^2})\in \mathbb{R}^2, t\in (-1/2,1/2)\}$. Fix one $x\in\mathcal{S}(\mathcal{R}_\theta^2)$. By duality and the H\"{o}lder inequality, we have
   \begin{align*}
          \|\hat{x}\|_{L_q(S_+^1)}&=\sup_{g\in C(S_+^1),\|g\|_{L_{q^\prime}(S_+^1)}=1}\left|\int_{S_+^1}\hat{x}(\xi)g(\xi)d\sigma(\xi)\right|\\
   &=\sup_{g\in C(S_+^1),\|g\|_{L_{q^\prime}(S_+^1)}=1}\left|\int_{S_+^1}\tau_\theta(xU_\theta(\xi)^*)g(\xi)d\sigma(\xi)\right|\\
   &=\sup_{g\in C(S_+^1),\|g\|_{L_{q^\prime}(S_+^1)}=1}\left|\tau_\theta(\int_{S_+^1}xU_\theta(\xi)^*g(\xi)d\sigma(\xi))\right|\\
   &\leq \|x\|_{L_p(\mathcal{R}_\theta^2)}\sup_{g\in C(S_+^1),\|g\|_{L_{q^\prime}(S_+^1)}=1}\Big\|\int_{S_+^1}U_\theta(\xi)^*g(\xi)d\sigma(\xi)\Big\|_{L_{p^\prime}(\mathcal{R}_\theta^2)}.
        \end{align*}
  Therefore, it suffices to show that for all $g\in C(S_+^1)$,
  \beq\label{du}\Big\|\int_{S_+^1}U_\theta^*(\xi)g(\xi)d\sigma(\xi)\Big\|_{L_{p^\prime}(\mathcal{R}_\theta^2)}\lesssim_{p}\|g\|_{L_{q^\prime}(S_+^1)}.\eeq
   We define $\gamma:(-1/2,1/2)\rightarrow\mathbb{R}$ as $\gamma(t):=\sqrt{1-t^2}$, and $\Gamma:(-1/2,1/2)\rightarrow\mathbb{R}^2$ as $\Gamma(t):=(t,\gamma(t)).$ Then for a fixed $g\in C(S_+^1)$,
   \begin{align*}
    & \qquad\Big|\int_{S_+^1}U_\theta^*(\xi)g(\xi)d\sigma(\xi)\Big|^2\\&=\Big|\int^{1/2}_{-1/2}U_\theta^*(\Gamma(t))g(\Gamma(t))(1+|\gamma^\prime(t)|^2)^{1/2}d t\Big|^2\\
    &\thickapprox\Big|\int^{1/2}_{-1/2}U_\theta^*(\Gamma(t))g(\Gamma(t))\gamma(t)^{-1}dt\Big|^2\\
    &=\left(\int_{-1/2}^{1/2}U_\theta(\Gamma(t))\bar{g}(\Gamma(t))\gamma(t)^{-1}dt\right)\left(\int_{-1/2}^{1/2}U_\theta^*(\Gamma(s))g(\Gamma(s))\gamma(s)^{-1}ds\right)\\
    &=\int_{-1/2}^{1/2}\int_{-1/2}^{1/2}\bar{g}(\Gamma(t))g(\Gamma(s))e^{-\frac{i}{2}(\Gamma(s),\theta\Gamma( t))}\gamma(t)^{-1}\gamma(s)^{-1}U_\theta(\Gamma(t)-\Gamma(s))dsdt.
   \end{align*}
   Define $T:(-1/2,1/2)\times(-1/2,1/2)\rightarrow \mathbb{R}^2$ as $T(t,s)=\Gamma(t)-\Gamma(s)\triangleq \eta$, it is easy to check that $T$ is injective.
   Let $B$ be the image of $T$, then $T:(-1/2,1/2)\times(-1/2,1/2)\rightarrow B$ is  invertible, and the Jacobian of $T$ is
   $J(t,s)=|\gamma^\prime(t)-\gamma^\prime(s)|\geq c|t-s|$, where $c=\min_{v\in(-1/2,1/2)}|\gamma^{\prime\prime}(v)|$.
   Then we have
  \beqq
    \Big|\int_{S_+^1}U_\theta^*(\xi)g(\xi)d\sigma(\xi)\Big|^2=\int_{\mathbb{R}^2}U_\theta(\eta)F(\eta)d\eta=U_\theta(F),
\eeqq
 where $F(T(t,s))=e^{-\frac{i}{2}(\Gamma(s),\theta\Gamma( t))}\gamma(t)^{-1}\gamma(s)^{-1}\bar{g}(\Gamma(t))g(\Gamma(s))J^{-1}(t,s)\chi_B(T(t,s))$.
  Let $r=p^\prime/2$, then $r>2$.  By Lemma \ref{HY}, \beqq\Big\|\int_{S_+^1}U_\theta^*(\xi)g(\xi)d\sigma(\xi)\Big\|_{L_{p^\prime}(\mathcal{R}_\theta^2)}^2=\|U_\theta(F)\|_{L_{r}(\mathcal{R}_\theta^2)}\leq\|F\|_{r^\prime}.\eeqq
To continue, combing the estimate of $J(t,s)$ and the H\"{o}lder inequality, we have
  \begin{align*}
    \|F\|_{r^\prime}^{r^\prime} &=\int_{-1/2}^{1/2}\int_{-1/2}^{1/2}|\gamma(t)^{-1}\gamma(s)^{-1}\bar{g}(\Gamma(t))g(\Gamma(s))|^{r^\prime}J^{-r^\prime+1}(t,s)dtds\\
     &\lesssim \int_{-1/2}^{1/2}\int_{-1/2}^{1/2}|\gamma(t)^{-1}\gamma(s)^{-1}\bar{g}(\Gamma(t))g(\Gamma(s))|^{r^\prime}|t-s|^{(1-r^\prime)}dtds\\
   &\lesssim\||\gamma(t)^{-1}g(\Gamma(t))|^{r^\prime}\|_{q^\prime/r^\prime}\left\|\int_{-1/2}^{1/2}|\gamma(s)^{-1}g(\Gamma(s))|^{r^\prime}|t-s|^{(1-r^\prime)}ds\right\|_{(q^\prime/r^\prime)^\prime}\\
    &\lesssim\||\gamma(t)^{-1}g(\Gamma(t))|^{r^\prime}\|_{q^\prime/r^\prime}^2\\
     &\lesssim\left(\int_{-1/2}^{1/2}|g(\Gamma(t))|^{q^\prime}\gamma(t)^{-1}dt\right)^{2r^\prime/q^\prime}=\|g\|_{L_{q^\prime}(S_+^1)}^{2r^\prime}.
  \end{align*}
Here we have used the fact $\gamma(t)^{-1}\thickapprox1$ when $t\in(-1/2,1/2)$  and the estimate
  \begin{align*}
   \left\|\int_{-1/2}^{1/2}|\gamma(s)^{-1}g(\Gamma(s))|^{r^\prime}|t-s|^{(1-r^\prime)}ds\right\|_{(q^\prime/r^\prime)^\prime}&\lesssim\||\gamma(t)^{-1}g(\Gamma(t))|^{r^\prime}\|_{q^\prime/r^\prime},
  \end{align*}
  which follows from  the Hardy-Littlewood-Sobolev inequality since $1+\frac{1}{(q^\prime/r^\prime)^\prime}= r^\prime-1+\frac{1}{q^\prime/r^\prime}$.
Therefore, the proof of Theorem \ref{rt1} is completed.
 \end{proof}
 \section{The proof of Theorem \ref{res} }\label{s4}

We firstly give the proof of Theorem \ref{res} in the non-endpoint case $1\leq p<4/3$ and $q=\frac{p^\prime}{3}$.
\begin{proof}
Let $x\in\mathcal{S}(\mathcal{R}_\theta^2)$. Recalling that $\chi^\delta(\xi)=\chi_{(1-\delta,1+\delta)}(|\xi|)$, we have
\begin{align*}
\|\chi^\delta \hat{x}\|_{L_q(\mathbb{R}^2)}^q&=\int_{1-\delta\leq|\xi|\leq1+\delta}|\hat{x}(\xi)|^qd\xi \\
  &=\int_{1-\delta}^{1+\delta}\int_{rS^1}|\hat{x}(\xi)|^qd\sigma(\xi)dr\\&=\int_{1-\delta}^{1+\delta}r\int_{S^1}|\hat{x}(r\xi)|^qd\sigma(\xi)dr.
\end{align*}
Let $T$ be the dilation map on $\mathbb{R}^2$, that is $T(\xi)=r^{-1}\xi$ for $\xi\in \mathbb{R}^2$.
Since $1\leq p<4/3$ and $q=\frac{p^\prime}{3}$, by Theorem \ref{rt1} and applying Proposition \ref{trans} to $T$,
\begin{align*}
 \|\chi^\delta \hat{x}\|_{L_q(\mathbb{R}^2)}^q &=\int_{1-\delta}^{1+\delta}r\int_{S^1}|\hat{x}(r\xi)|^qd\sigma(\xi)dr \\
 &\lesssim \int_{1-\delta}^{1+\delta}r \|U_{\theta_{T^{-1}}}(\hat{x}(r\cdot)\|_{L_{p}(\mathcal{R}_{\theta_{T^{-1}}}^2)}^qdr\\
 &=\int_{1-\delta}^{1+\delta}r^{1+\frac{2q}{p}} \|\Psi_{T}(U_{\theta_{T^{-1}}}(\hat{x}(r\cdot))\|_{L_{p}(\mathcal{R}_{\theta}^2)}^qdr\\
 &=\int_{1-\delta}^{1+\delta}r^{1+\frac{2q}{p}-2q} \|x\|_{L_{p}(\mathcal{R}_{\theta}^2)}^qdr\\
 &=\int_{1-\delta}^{1+\delta}r^{\frac{1}{3}} \|x\|_{L_{p}(\mathcal{R}_{\theta}^2)}^qdr\\
 &\lesssim \delta \|x\|_{L_{p}(\mathcal{R}_\theta^2)}^q,
\end{align*}
where the final inequality follows from the fact that $\delta\in(0,1/2)$.
The proof of the first part of Theorem \ref{res} is now complete.
\end{proof}
\begin{remark}{\rm
  Actually, Theorem \ref{res} can conversely imply Theorem \ref{rt1} once we note that
  \begin{align*}
     \|\hat{x}\|_{L_q(S^1)}^q&=\lim_{\delta\rightarrow0}\frac{1}{2\delta}\int_{1-\delta}^{1+\delta}\int_{0}^{2\pi}r|\hat{x}(r\omega)|^qd\omega dr\\
    &=\lim_{\delta\rightarrow0}\frac{1}{2\delta}\int_{0}^{\infty}\int_{0}^{2\pi}r|\chi^\delta(r\omega)\hat{x}(r\omega)|^qd\omega dr\\
    &=\lim_{\delta\rightarrow0}\frac{1}{2\delta}\|\chi^\delta\hat{x}\|_{L_q(\mathbb{R}^2)}^q.
  \end{align*}}
\end{remark}
Now we will focus on the proof of the endpoint case: $p=4/3$.
That is, for all $x\in \mathcal{S}(\mathcal{R}_\theta^2)$
  \beq\label{end}
  \|\chi^\delta \hat{x}\|_{L_{4/3}(\mathbb{R}^2)}\lesssim\delta ^{3/4}(\log{\delta^{-1}})^{1/4}\|x\|_{L_{4/3}(\mathcal{R}_\theta^2)}.\eeq
 \begin{proof}[Proof of \eqref{end}]
Let $x\in \mathcal{S}(\mathcal{R}_\theta^2)$.
By duality, we have
 \beqq
   \|\chi^\delta \hat{x}\|_{L_{4/3}(\mathbb{R}^2)}=\sup_{\|g\|_{L_{4}(\mathbb{R}^2)}=1}\left|\int_{\mathbb{R}^2}\chi^\delta(\xi)\hat{x}(\xi)g(-\xi)d\xi\right|.\eeqq
  By Lemma \ref{f*g}
  \beqq
  \int_{\mathbb{R}^2}\chi^\delta(\xi)\hat{x}(\xi)g(-\xi)d\xi=\tau_\theta( U_\theta(\hat{x}) U_\theta(\chi^\delta g)).\eeqq
 Then we get
   \beqq
 \|\chi^\delta \hat{x}\|_{L_{4/3}(\mathbb{R}^2)}=\sup_{\|g\|_{L_{4}(\mathbb{R}^2)}=1}\left|\tau_\theta( U_\theta(\hat{x}) U_\theta(\chi^\delta g))\right|\leq \|x\|_{L_{4/3}(\mathcal{R}_\theta^2)}\sup_{\|g\|_{L_{4}(\mathbb{R}^2)}=1}\|U_\theta(\chi^\delta g)\|_{L_{4}(\mathcal{R}_\theta^2)}.\eeqq
  Hence it suffices to show the following estimate for all $g\in{L_{4}(\mathbb{R}^2)}$,
 \beq\label{rrr2}
  \|U_\theta(\chi^\delta g)\|_{L_{4}(\mathcal{R}_\theta^2)}\lesssim \delta ^{3/4}(\log{\delta^{-1}})^{1/4}\|g\|_{L_{4}(\mathbb{R}^2)}\eeq
for sufficiently small $\delta$. For an integer $\ell\in \{0,1,...,[\delta^{-1/2}]\}$, we set
   \beqq
   \chi_\ell^\delta(\xi):=\chi^\delta(\xi)\chi_{2\pi \ell\delta^{1/2}\leq \arg \xi<2\pi(\ell+1)\delta^{1/2}}.\eeqq
   Here we suitably adjust the support of $\chi_{[\delta^{-1/2}]}^\delta$ so that $\sum_{l=0}^{[\delta^{-1/2}]}\chi_\ell^\delta(\xi)=\chi^\delta(\xi)$.
   We now split $\{0,1,...,[\delta^{-1/2}]\}$ into nine different subsets so that the supports of the functions with indices in each subset are contained in some sector centered at the origin of amplitude $\pi/4$. Let us fix one index set $I$. Without loss of generality we assume that
   \beqq
   I=\{0,1,...,[1/8\delta^{-1/2}]\}.\eeqq
   Given $g \in L_{4}(\mathbb{R}^2)$, it is reduced to showing $$\left\|U_\theta\Big(\sum_{\ell\in I}\chi_\ell^\delta g\Big)\right\|_{L_{4}(\mathcal{R}_\theta^2)}\lesssim \delta ^{3/4}(\log{\delta^{-1}})^{1/4}\|g\|_{L_{4}(\mathbb{R}^2)}.$$
   Without loss of generality, we may assume that $g$ is real-valued.
   Note that
   \begin{align}\label{for rem}
    \left\|U_\theta\Big(\sum_{\ell\in I}\chi_\ell^\delta g\Big)\right\|_{L_{4}(\mathcal{R}_\theta^2)}^{4}
    &=\tau_\theta\left(\Big|U_\theta\Big(\sum_{\ell\in I}\chi_\ell^\delta g\Big)^*U_\theta\Big(\sum_{\ell'\in I}\chi_{\ell'}^\delta g\Big)\Big|^{2}\right)\\
    &\nonumber\lesssim (I)+(II),
   \end{align}
 where
 \beqq (I)= \tau_\theta\left(\Big|\sum_{\ell,\ell^\prime\in I, |\ell-\ell^\prime|\leq 10^3}U_\theta(\chi_\ell^\delta g)^*U_\theta(\chi_{\ell^\prime}^\delta g)\Big|^{2}\right)\eeqq and
  \beqq (II)= \tau_\theta\left(\Big|\sum_{\ell,\ell^\prime\in I, |\ell-\ell^\prime|> 10^3}U_\theta(\chi_\ell^\delta g)^*U_\theta(\chi_{\ell^\prime}^\delta g)\Big|^{2}\right).\eeqq

   We first estimate part $(I)$.
   Let $\mathcal{N}=\mathcal{R}_\theta^2\bar{\otimes}\mathcal{B}(\ell_2(I))$, and $v=\tau_\theta\otimes \mathrm{Tr}$ be the trace on $\mathcal{N}$, where $\mathrm{Tr}$ is the usual trace on $\mathcal{B}(\ell_2(I))$. Applying the noncommutative H\"{o}lder inequality,
   \begin{align*}
     &\tau_\theta\left(\Big| \sum_{\ell\in I,\ell+i\in I}U_\theta(\chi_\ell^\delta g)^*U_\theta(\chi_{\ell+i}^\delta g)\Big|^{2}\right)\\ &= v\left(\Big|\Big(\sum_{\ell\in I,\ell+i\in I}U_\theta(\chi_\ell^\delta g)\otimes e_{\ell,1}\Big)^*\Big( \sum_{\ell\in I,\ell+i\in I}U_\theta(\chi_{\ell+i}^\delta g)\otimes e_{\ell,1}\Big)\Big|^2\right)
     \\
   &\leq\Big\|\sum_{\ell\in I,\ell+i\in I}U_\theta(\chi_\ell^\delta g)\otimes e_{\ell,1}\Big\|_{L_4(\mathcal{N})}^2\Big\|\sum_{\ell\in I, \ell+i\in I}U_\theta(\chi_{\ell+i}^\delta g)\otimes e_{\ell,1}\Big\|_{L_4(\mathcal{N})}^2\\
   &= \Big\|\Big( \sum_{\ell\in I,\ell+i\in I}|U_\theta(\chi_\ell^\delta g)|^2\Big)^{1/2}\Big\|_{L_4(\mathcal{R}_\theta^d)}^2\Big\|\Big( \sum_{\ell\in I,\ell+i\in I}|U_\theta(\chi_{\ell+i}^\delta g)|^2\Big)^{1/2}\Big\|_{L_4(\mathcal{R}_\theta^d)}^2.
   \end{align*}
Here $i\in \mathbb Z$ and $e_{\ell,1}$ is the matrix which has a unique non-zero entry $1$ at the position $(\ell,1)$. Then by the noncommutative Minkowski inequality,
 \begin{align*}
     (I)&=\tau_\theta\left(\Big|\sum_{i=-10^3}^{10^3} \sum_{\ell\in I,\ell+i\in I}U_\theta(\chi_\ell^\delta g)^*U_\theta(\chi_{\ell+i}^\delta g)\Big|^{2}\right)\\
   &\lesssim \sum_{i=-10^3}^{10^3}\tau_\theta\left(\Big| \sum_{\ell\in I,\ell+i\in I}U_\theta(\chi_\ell^\delta g)^*U_\theta(\chi_{\ell+i}^\delta g)\Big|^{2}\right)\\
   &\lesssim \sum_{i=-10^3}^{10^3}\Big\|\Big( \sum_{\ell\in I,\ell+i\in I}|U_\theta(\chi_\ell^\delta g)|^2\Big)^{1/2}\Big\|_4^2\Big\|\Big( \sum_{\ell\in I,\ell+i\in I}|U_\theta(\chi_{\ell+i}^\delta g)|^2\Big)^{1/2}\Big\|_4^2\\
   &\lesssim \tau_\theta\left(\Big( \sum_{\ell\in I}|U_\theta(\chi_\ell^\delta g)|^2\Big)^{2}\right)=\tau_\theta\Big( U_\theta\Big(\sum_{\ell\in I}(\tilde{\chi}_\ell^\delta \tilde{g})*_\theta(\chi_\ell^\delta g)\Big)^{2}\Big).\\
   \end{align*}
   Applying Lemma \ref{HY},
   \begin{align*}
    \tau_\theta\Big( U_\theta\Big(\sum_{\ell\in I}(\tilde{\chi}_\ell^\delta \tilde{g})*_\theta(\chi_\ell^\delta g)\Big)^{2}\Big)&=\Big\|\sum_{\ell\in I}(\tilde{\chi}_\ell^\delta \tilde{g})*_\theta(\chi_\ell^\delta g)\Big\|_{L_2(\mathbb{R}^2)}^2\\ &\leq\Big\|\sum_{\ell\in I}|\tilde{\chi}_\ell^\delta \tilde{g}|*|\chi_\ell^\delta g|\Big\|_{L_2(\mathbb{R}^2)}^2.\\
   \end{align*}
   We claim that
   \begin{align}\label{2r}\Big\|\sum_{\ell\in I}|\tilde{\chi}_\ell^\delta \tilde{g}|*|\chi_\ell^\delta g|\Big\|_{L_2(\mathbb{R}^2)} =\Big\|\Big(\sum_{\ell,\ell^\prime\in I}(|\tilde{\chi}_\ell^\delta \tilde{g}|*|\chi_{\ell^\prime}^\delta{g}|)^2\Big)^{1/2}\Big\|_{L_2(\mathbb{R}^2)}.
   \end{align}
  Indeed, by the two elementary facts that $\int_{\mathbb{R}^2}f(s)(g*h)(s)ds= \int_{\mathbb{R}^2}(f*\tilde{g})(s)h(s)ds$ and $f*g=g*f$,
   \begin{align*}
    \Big\|\sum_{\ell\in I}|\tilde{\chi}_\ell^\delta \tilde{g}|*|\chi_\ell^\delta g|\Big\|_{L_2(\mathbb{R}^2)}^2&=\int_{\mathbb{R}^2}\left(\sum_{\ell\in I}(|\tilde{\chi}_\ell^\delta \tilde{g}|*|\chi_{\ell}^\delta{g}|)(s)\right)^2ds\\
  &=\int_{\mathbb{R}^2}\sum_{\ell,\ell^\prime\in I}(|\tilde{\chi}_\ell^\delta \tilde{g}|*|\chi_\ell^\delta{g}|)(s)(|\tilde{\chi}_{\ell^\prime}^\delta \tilde{g}|*|\chi_{\ell^\prime}^\delta{g}|)(s)ds\\
  &=\Big\|\Big(\sum_{\ell,\ell^\prime\in I}(|\tilde{\chi}_{\ell}^\delta \tilde{g}|*|\chi_{\ell^\prime}^\delta{g}|)^2\Big)^{1/2}\Big\|_{L_2(\mathbb{R}^2)}^2.
   \end{align*}   Therefore, we get
    \beq\label{rrr5}
    (I)\lesssim\Big\|\Big(\sum_{\ell,\ell^\prime\in I}(|\tilde{\chi}_\ell^\delta \tilde{g}|*|\chi_{\ell^\prime}^\delta{g}|)^2\Big)^{1/2}\Big\|_{L_2(\mathbb{R}^2)}^{2}.\eeq

     Now we deal with part $(II)$. Setting $S_{\delta,\ell,\ell^\prime}=\text{supp}(\tilde{\chi}_{\ell^\prime}^\delta *\chi_{\ell^\prime}^\delta)$, by the Plancherel theorem and the H\"older inequality, one has
    \begin{align*}
    (II) &= \int_{\mathbb{R}^2}\Big|\sum_{\ell,\ell^\prime\in I, |\ell-\ell^\prime|> 10^3}(\tilde{\chi}_\ell^\delta \tilde{g})*_\theta(\chi_{\ell^\prime}^\delta g)\Big|^2ds \\
    &\lesssim \int_{\mathbb{R}^2}\Big(\sum_{\ell,\ell^\prime\in I, |\ell-\ell^\prime|> 10^3}|\tilde{\chi}_\ell^\delta \tilde{g}|*|\chi_{\ell^\prime}^\delta g|\Big)^2ds\\
    &\lesssim \int_{\mathbb{R}^2}\Big(\sum_{\ell,\ell^\prime\in I, |\ell-\ell^\prime|> 10^3}(|\tilde{\chi}_\ell^\delta \tilde{g}|*|\chi_{\ell^\prime}^\delta g|)^2\Big)\Big(\sum_{\ell,\ell^\prime\in I, |\ell-\ell^\prime|> 10^3}\chi_{S_{\delta,\ell,\ell^\prime}}\Big)ds.
    \end{align*}
    Using  Lemma 5.4 in \cite{Lai21}, which states that for $s\in\mathbb{R}^2$, there is a constant $C$ such that
    \beq\label{laige}
    \sum_{\ell,\ell^\prime\in I, |\ell-\ell^\prime|> 10^3}\chi_{S_{\delta,\ell,\ell^\prime}}(s)\leq C.\eeq
    Thus we have
    \beq\label{rrr6}
    (II)\lesssim \Big\|\Big(\sum_{\ell,\ell^\prime\in I}(|\tilde{\chi}_\ell^\delta \tilde{g}|*|\chi_{\ell^\prime}^\delta{g}|)^2\Big)^{1/2}\Big\|_{L_2(\mathbb{R}^2)}^{2}.\eeq
    By \eqref{rrr5} and \eqref{rrr6}
    \beq\label{rrr7}
    \left\|U(\sum_{\ell\in I}\chi_\ell^\delta g)\right\|_{L_{4}(\mathcal{R}_\theta^2)}^4\lesssim \Big\|\Big(\sum_{\ell,\ell^\prime\in I}(|\tilde{\chi}_\ell^\delta \tilde{g}|*|\chi_{\ell^\prime}^\delta{g}|)^2\Big)^{1/2}\Big\|_{L_2(\mathbb{R}^2)}^2=\sum_{\ell,\ell^\prime\in I}\Big\||\tilde{\chi}_\ell^\delta \tilde{g}|*|\chi_{\ell^\prime}^\delta{g}|\Big\|_{L_2(\mathbb{R}^2)}^2.\eeq
Hence it is reduced to  showing \beq\label{last}
\sum_{\ell,\ell^\prime\in I}\Big\||\tilde{\chi}_\ell^\delta \tilde{g}|*|\chi_{\ell^\prime}^\delta{g}|\Big\|_{L_2(\mathbb{R}^2)}^2
\lesssim \delta ^3(\log{\delta^{-1}})\|g\|_{L_{4}(\mathbb{R}^2)}^4.\eeq

    We need a lemma similar to Lemma 5.4.8 in \cite{GF}.
    \begin{lemma}\label{rrrl1}
    For $1\leq r\leq\infty$, there is a constant $C$ which is independent of $\delta$ and $g$ such that
    \beqq
    \Big\||\tilde{\chi}_\ell^\delta \tilde{g}|*|\chi_{\ell^\prime}^\delta{g}|\Big\|_{L_r(\mathbb{R}^2)}\leq C\left(\frac{\delta^{3/2}}{|\ell-\ell^\prime|+1}\right)^{\frac{1}{r^\prime}}\|\chi_\ell^\delta g\|_{L_r(\mathbb{R}^2)}\|\chi_{\ell^\prime}^\delta{g}\|_{L_r(\mathbb{R}^2)}.
    \eeqq
    \end{lemma}
    \begin{proof}
     For $\ell,\ell^\prime\in I$, we define the bilinear operator
    \beqq
    Bi_{\ell,\ell^\prime}(g,h):=(\tilde{\chi}_\ell^\delta \tilde{g})*(\chi_{\ell^\prime}^\delta{h}).\eeqq
    It suffices to prove the following estimate
    \beq
    \Big\|Bi_{\ell,\ell^\prime}(g,h)\Big\|_{L_r(\mathbb{R}^2)}\lesssim\left(\frac{\delta^{3/2}}{|\ell-\ell^\prime|+1}\right)^{\frac{1}{r^\prime}}\| g\|_{L_r(\mathbb{R}^2)}\| h\|_{L_r(\mathbb{R}^2)}.
    \eeq
    By the construction of $\chi_\ell^\delta$, $\text{supp}(\tilde{\chi}_\ell^\delta)$ is contained in a rectangle of dimensions $\thicksim \delta\times\delta^{1/2}$ 
     with the direction of the width being $e^{i2\pi\ell\delta^{1/2}}$ and the direction of the length $ie^{i2\pi\ell\delta^{1/2}}$.
      While $\text{supp}({\chi_{\ell^\prime}^\delta})$ is contained in a rectangle of dimensions $\thicksim \delta\times\delta^{1/2}$, 
     with the direction of the width being $e^{i2\pi\ell^\prime\delta^{1/2}}$ and the direction of the length $ie^{i2\pi\ell^\prime\delta^{1/2}}$.

     We claim that for a given $z\in\mathbb{R}^2$, $(z-\text{supp}(\tilde{\chi}_\ell^\delta))\cap\text{supp}(\chi_{\ell^\prime}^\delta)$ is contained in a rectangle of dimensions $\thicksim \delta\times\frac{\delta^{1/2}}{1+|\ell-\ell^\prime|}$ .
    Indeed, we may assume that the intersection is not empty. Let $x,y\in (z-\text{supp}(\tilde{\chi}_\ell^\delta))\cap\text{supp}(\chi_{\ell^\prime}^\delta)$,
    then $x-y$ has the form of $t_1e^{i2\pi\ell\delta^{1/2}}+t_2ie^{i2\pi\ell\delta^{1/2}}$. On one hand, $x,y\in (z-\text{supp}(\tilde{\chi}_\ell^\delta))$ implies that $|t_1|\lesssim\delta$ and $|t_2|\lesssim\delta^{1/2}$. On the other hand,
    $x,y\in \text{supp}(\chi_{\ell^\prime}^\delta)$ implies that  $|\langle x-y, e^{i2\pi\ell^\prime\delta^{1/2}}\rangle|\lesssim\delta $. Thus
    \beqq |2t_2\sin(2\pi\delta^{1/2}|\ell-\ell^\prime|)|=|\langle t_2ie^{i2\pi\ell\delta^{1/2}}, e^{i2\pi\ell^\prime\delta^{1/2}}\rangle|\leq|t_1|+|\langle x-y, e^{i2\pi\ell^\prime\delta^{1/2}}\rangle|\lesssim\delta.\eeqq
    Note that $2\pi\delta^{1/2}|\ell-\ell^\prime|<\pi/4$, so $ \sin(2\pi\delta^{1/2}|\ell-\ell^\prime|)\thicksim \delta^{1/2}|\ell-\ell^\prime|$. And thus we have
    $|t_2|\lesssim \frac{\delta^{1/2}}{|\ell-\ell^\prime|}$,
 which, together with $|t_2|\lesssim\delta^{1/2}$, yields the claim.

    By the claim, \beqq
      \|\tilde{\chi}_\ell^\delta*\chi_{\ell^\prime}^\delta\|_\infty\leq \sup_{z\in\mathbb{R}^2}|(z-\text{supp}(\tilde{\chi}_\ell^\delta))\cap\text{supp}(\chi_{\ell^\prime}^\delta)|\lesssim\frac{\delta^{3/2}}{|\ell-\ell^\prime|+1}.\eeqq
      Thus we have \beq\label{i1}
    \Big\|Bi_{\ell,\ell^\prime}(g,h)\Big\|_{L_\infty(\mathbb{R}^2)}\leq \|\tilde{\chi}_\ell^\delta*\chi_{\ell^\prime}^\delta\|_\infty\|g\|_\infty\|h\|_\infty\lesssim\frac{\delta^{3/2}}{|\ell-\ell^\prime|+1}\|g\|_\infty\|h\|_\infty.\eeq
    Note that we also have
    \beq\label{i2}
    \Big\|Bi_{\ell,\ell^\prime}(g,h)\Big\|_{L_1(\mathbb{R}^2)}\leq \|g\|_{L_1(\mathbb{R}^2)}\|h\|_{L_1(\mathbb{R}^2)}.\eeq
      By interpolating \eqref{i1} and \eqref{i2}, we have \beqq
    \Big\|Bi_{\ell,\ell^\prime}(g,h)\Big\|_{L_r(\mathbb{R}^2)}\lesssim\left(\frac{\delta^{3/2}}{|\ell-\ell^\prime|+1}\right)^{\frac{1}{r^\prime}}\| g\|_{L_r(\mathbb{R}^2)}\| h\|_{L_r(\mathbb{R}^2)}.
    \eeqq

    \end{proof}

    Now we may conclude the proof. By the H\"{o}lder inequality and the Young inequality
 for discrete $L_p$ spaces, and applying Lemma \ref{rrrl1} to $r=2$, one gets  \begin{align*}
    \sum_{\ell,\ell^\prime\in I}\Big\||\tilde{\chi}_\ell^\delta \tilde{g}|*|\chi_{\ell^\prime}^\delta{g}|\Big\|_{L_2(\mathbb{R}^2)}^2&\lesssim
      \delta^{\frac{3}{2}}\left(\sum_{\ell,\ell^\prime\in I}\|\chi_\ell^\delta g\|_{L_2(\mathbb{R}^2)}^2\frac{\|\chi_{\ell^\prime}^\delta{g}\|_{L_2(\mathbb{R}^2)}^2}{(|\ell-\ell^\prime|+1)}\right)\\
   &\lesssim
      \delta^{\frac{3}{2}}\left(\sum_{\ell\in I}\|\chi_\ell^\delta g\|_{L_2(\mathbb{R}^2)}^4\right)^{1/2}\left[\sum_{\ell\in I}\left(\sum_{\ell^\prime\in I}\frac{\|\chi_{\ell^\prime}^\delta{g}\|_{L_2(\mathbb{R}^2)}^2}{(|\ell-\ell^\prime|+1)}\right)^2\right]^{1/2}\\
     &\lesssim
      \delta^{\frac{3}{2}}\left(\sum_{\ell\in I}\|\chi_\ell^\delta g\|_{L_2(\mathbb{R}^2)}^4\right)^{1/2}\left(\sum_{\ell\in I}\|\chi_\ell^\delta g\|_{L_2(\mathbb{R}^2)}^4\right)^{1/2}\left(\sum_{\ell\in I}\frac{1}{|\ell|+1}\right)\\
      &\lesssim
      \delta^{\frac{3}{2}}(\log\delta^{-1})\left(\sum_{\ell\in I}\|\chi_\ell^\delta g\|_{L_2(\mathbb{R}^2)}^4\right).
  \end{align*}
  Then by the H\"{o}lder inequality,  $\|\chi_\ell^\delta g\|_{L_2(\mathbb{R}^2)}\leq \delta^{\frac{3}{2}\cdot(1/2-1/4)}\|\chi_\ell^\delta g\|_{L_4(\mathbb{R}^2)}. $ Hence we have
  \beqq
  \sum_{\ell,\ell^\prime\in I}\Big\||\tilde{\chi}_\ell^\delta \tilde{g}|*|\chi_{\ell^\prime}^\delta{g}|\Big\|_{L_2(\mathbb{R}^2)}^2\lesssim \delta ^3(\log{\delta^{-1}})\left(\sum_{\ell\in I}\|\chi_\ell^\delta g\|_{L_4(\mathbb{R}^2)}^4\right)\lesssim \delta ^3(\log{\delta^{-1}})\|g\|_{L_{4}(\mathbb{R}^2)}^4.\eeqq
\end{proof}
\begin{remark}\label{diff}{\rm
(i). The above arguments for the endpoint case $p=4/3$ is motivated by the classical ones (see e.g. Chapter 5 in \cite{GF}). But the noncommutativity makes the arguments much more involved. For instance, due to the fact that $|x^*x|^2$ is not necessarily equal to $|xx|^2$ for a general operator $x$, we have to separate the double sums over $\ell$ and $\ell'$ in  \eqref{for rem} according to size of the difference $|\ell-\ell'|$ to avoid the uncontrollable overlapping of  $\{{S_{\delta,\ell,\ell^\prime}}\}_{\ell,\ell'\in I}$ around the origin; this in turn yields some new geometric arguments such as the ones in Lemma 5.4 in \cite{Lai21} and Lemma \ref{rrrl1}.

(ii). In the case $\theta=0$, the arguments for the endpoint case $p=4/3$ works also for the other cases $1\leq p<4/3$ (see e.g. Chapter 5 in \cite{GF}). Nevertheless, in the present noncommutative setting $\theta\neq0$, we came across some difficulties in adapting the proof of the endpoint estimate \eqref{rrr1} for the non-endpoint one \eqref{rrr}. For instance, until the moment of writing, we do not know how to obtain the following variant of \eqref{2r} where $r=2$,
  \begin{equation}\label{ques}
  \Big\|\sum_{\ell\in I}|\tilde{\chi}_\ell^\delta \tilde{g}|*|\chi_\ell^\delta g|\Big\|_{L_r(\mathbb{R}^2)}\lesssim
 \Big\|\Big(\sum_{\ell,\ell^\prime\in I}(|\tilde{\chi}_\ell^\delta \tilde{g}|*|\chi_{\ell^\prime}^\delta{g}|)^r\Big)^{1/r}\Big\|_{L_r(\mathbb{R}^2)}
  \end{equation}
 for all $1< r<2$. Note that the estimate is trivial when $r=1$.}
\end{remark}

\section{The proof of Theorem \ref{rt2}}\label{s5}
Before the proof of Theorem \ref{rt2}, we show a type of the Young inequality on $\mathcal{R}_\theta^d$ via some properties of operator spaces, for which we refer to \cite{Pi}. Recall that the opposite algebra of quantum Euclidean space $(\mathcal{R}_\theta^d)_{op}$ is obtained by preserving linear and adjoint structures but reversing the product, i.e. for $x,y\in (\mathcal{R}_\theta^d)_{op}$, the product $\cdot$ is defined by
  $x\cdot y=yx.$ As in \cite{GJP}, one can define a normal $*$-homomorphism map $\pi_\theta$  from $L_\infty(\mathbb{R}^d)$ to $\mathcal{R}_\theta^d\bar{\otimes}(\mathcal{R}_\theta^d)_{op}$, which is determined by
  \beqq \pi_\theta(\exp_t):=U_\theta(t)\otimes U_\theta(t)^*.\eeqq

We have the following lemma which should be known to experts. For  the sake of completeness, we give the details of the proof.
\begin{lemma}\label{1-0}
For $\psi\in L_1(\mathbb{R}^d)$ and $x\in\mathcal{S}(\mathcal{R}_\theta^d)$, let $T_\psi(x):=U_\theta(\psi \hat{x})$, then
\beqq \|T_\psi(x)\|_\infty\leq\|\check{\psi}\|_\infty\|x\|_1.\eeqq
Therefore, $T_\psi$ extends to a bounded linear operator from
$L_1(\mathcal{R}_\theta^d)$ to $L_\infty(\mathcal{R}_\theta^d)$.
\end{lemma}
\begin{proof}
  We claim that \beq\label{pi}
  T_\psi(x)=(id_{\mathcal{R}_\theta^d}\otimes\tau_\theta)(id_{\mathcal{R}_\theta^d}\otimes x)(\pi_\theta(\check{\psi})).\eeq
  Indeed,
   \begin{align*}
  (id_{\mathcal{R}_\theta^d}\otimes\tau_\theta)(id_{\mathcal{R}_\theta^d}\otimes x)(\pi_\theta(\check{\psi}))&= (id_{\mathcal{R}_\theta^d}\otimes\tau_\theta)\Big(\int_{\mathbb{R}^d}U_\theta(t)\otimes (x\psi(t)U_\theta(t)^*)dt\Big)\\
  &=\int_{\mathbb{R}^d} U_\theta(t)\tau_\theta(x\psi(t)U_\theta(t)^*)dt\\
 &=\int_{\mathbb{R}^d} U_\theta(t) \hat{x}(t)\psi(t)dt=T_{\psi}(x).
\end{align*}
Then noting that $(\mathcal{R}_\theta^d)_{op}$ is the dual space of $L_1(\mathcal{R}_\theta^d)$, by Theorem 2.5.2 in \cite{Pi}, one gets \beqq\|T_\psi\|_{L_1(\mathcal{R}_\theta^d)\rightarrow \mathcal{R}_\theta^d}\leq\|\pi_\theta(\check{\psi})\|_{\mathcal{R}_\theta^d\bar{\otimes}(\mathcal{R}_\theta^d)_{op}}\leq \|\check{\psi}\|_\infty.\eeqq

\end{proof}
Now we are at the position to prove Theorem \ref{rt2}.
\begin{proof}[Proof of Theorem \ref{rt2}]
  We first prove the case when $1\leq p<\frac{2(d+1)}{d+3}$.
  Let $\phi(s)$ be a bump function supported on the unit ball and $\phi(s)=1$ when $s\in B(0,1/2)$, $\phi_0(s):=\phi(s)$ and $\phi_k(s):=\phi({s}/{2^k})- \phi({s}/{2^{k-1}})$ when $k\geq1$, so that we have \beqq \sum_{k=0}^\infty \phi_k(s)=1, \quad \text{for all}\: s\in \mathbb{R}^d. \eeqq
    Setting $\check{d\sigma}(s):=\int_{\mathbb{R}^d} e^{2\pi i\langle s,\xi\rangle}d\sigma(\xi)$. For $x\in\mathcal{S}(\mathcal{R}_\theta^d)$, we may assume that $x=U_\theta(f)$ for some $f\in\mathcal{S}(\mathbb{R}^d)$,
 then by using twice the Parseval relation on $\mathbb{R}^d$ and Lemma \ref{f*g},
  \begin{align*}
    \|f\|_{L_2(S^{d-1},d\sigma)}^2&=\langle \check{f},  \check{f}*\check{d\sigma}\rangle_{L_2(\mathbb{R}^d)}\\&=\sum_{k=0}^\infty\langle \check{f},  \check{f}*(\phi_k\check{d\sigma})\rangle_{L_2(\mathbb{R}^d)} \\
    &=\sum_{k=0}^\infty\int_{\mathbb{R}^d}\overline{f(t)}\widehat{(\phi_k\check{d\sigma})}(t)f(t)dt\\
& =\sum_{k=0}^\infty\langle  U_\theta(f),U_\theta(\widehat{(\phi_k\check{d\sigma})}f)\rangle_{L_2(\mathcal{R}_\theta^d)}\\
&\leq \sum_{k=0}^\infty\|U_\theta(f)\|_{L_p(\mathcal{R}_\theta^d)}\Big\|U_\theta(\widehat{(\phi_k\check{d\sigma})}f)\Big\|_{L_{p^\prime}(\mathcal{R}_\theta^d)}.
  \end{align*}
  Thus it suffices to show that $\sum_{k=0}^\infty\Big\|U_\theta(\widehat{(\phi_k\check{d\sigma})}f)\Big\|_{L_{p^\prime}(\mathcal{R}_\theta^d)}\lesssim \|U_\theta(f)\|_{L_p(\mathcal{R}_\theta^d)}$ when $1\leq p< \frac{2(d+1)}{d+3}$.
  It is well known that  $|\check{d\sigma}(s)|\lesssim (1+|s|)^{\frac{-(d-1)}{2}}$ (see e.g. \cite{MS}), thus
 $\|\phi_k\check{d\sigma}\|_\infty\lesssim 2^{-k(d-1)/2}$. By Lemma \ref{1-0} we have
  \beq\label{l1}
  \Big\|U_\theta(\widehat{(\phi_k\check{d\sigma})}f)\Big\|_{L_{\infty}(\mathcal{R}_\theta^d)}\lesssim 2^{-k(d-1)/2}\|U_\theta(f)\|_{L_{1}(\mathcal{R}_\theta^d)}.
  \eeq
  On the other hand,  we claim that $\Big\|\widehat{\phi_k\check{d\sigma}}\Big\|_\infty\lesssim 2^k$. Indeed,
  for $\xi\in\mathbb{R}^d$
\beqq \widehat{\phi_k\check{d\sigma}}(\xi)=\int_{S^{d-1}}2^{kd}\hat{\phi}(2^k(\xi-\eta))d\sigma(\eta).\eeqq
Let $A=\{\eta\in S^{d-1}:|\xi-\eta|< 2^{-k}\}$ and $A_j=\{\eta\in S^{d-1}:2^{j-k}\leq|\xi-\eta|< 2^{j-k+1}\}$, where $j\geq0$. Then by  the rapid decay of $\hat{\phi}$,
\begin{align*}
 \Big|\widehat{\phi_k\check{d\sigma}}(\xi)\Big|&\leq \int_{A}2^{kd}|\hat{\phi}(2^k(\xi-\eta))|d\sigma(\eta)+ \sum_{j=0}^\infty\int_{A_j}2^{kd}|\hat{\phi}(2^k(\xi-\eta))|d\sigma(\eta)\\
 &\lesssim \int_{A}2^{kd}d\sigma(\eta)+\sum_{j=0}^\infty\int_{A_j}2^{kd}(2^k|\xi-\eta|)^{-100d}d\sigma(\eta)\\
 &\lesssim \sum_{j=0}^\infty 2^{kd}2^{-100jd}\sigma(B(\xi,2^{j-k}))
\lesssim 2^{k}.
\end{align*}
Here the final inequality follows from the fact that $\sigma(B(\xi,2^{j-k}))\lesssim\min\{1, 2^{(d-1)(j-k)}\}.$
  Then by Lemma \ref{HY}, one has
  \beq\label{l2}
\Big\|U_\theta(\widehat{(\phi_k\check{d\sigma})}f)\Big\|_{L_{2}(\mathcal{R}_\theta^d)}\lesssim2^k\|f\|_{L_2(\mathbb{R}^d)}\lesssim 2^k \|U_\theta(f)\|_{L_{2}(\mathcal{R}_\theta^d)}.\eeq
By interpolation between \eqref{l1} and \eqref{l2},
\beq\label{lp}
\Big\|U_\theta(\widehat{(\phi_k\check{d\sigma})}f)\Big\|_{L_{p^\prime}(\mathcal{R}_\theta^d)}\lesssim 2^{c(p,d)k}\|U_\theta(f)\|_{L_p(\mathcal{R}_\theta^d)},\eeq
 where $c(p,d)=(d+1)(1/2-1/p)+1$. When $p< \frac{2(d+1)}{d+3}$, we have $c(p,d)<0$. Then by \eqref{lp},
 $\sum_{k=0}^\infty\Big\|U_\theta(\widehat{(\phi_k\check{d\sigma})}f)\Big\|_{L_{p^\prime}(\mathcal{R}_\theta^d)}\lesssim \|U_\theta(f)\|_{L_p(\mathcal{R}_\theta^d)}$.

 Finally we consider the endpoint case,  $p=\frac{2(d+1)}{d+3}$. By a smoothing partition of unity, using the similar arguments as the ones in the proof of Proposition \ref{local} , it suffices to show that
 \beq\label{sp} \|f\|_{L_2(S^{d-1},\psi d\sigma)}\lesssim \|U_\theta(f)\|_{L_p(\mathcal{R}_\theta^d)}\eeq
 for a smoothing function $\psi$ supported on a neighborhood of the north pole $(0,0,...,0,1).$
 For $\xi=(\xi^\prime,\xi_d)\in\mathbb{R}^d$, where $\xi^\prime\in \mathbb{R}^{d-1}$, let $\varphi(\xi^\prime):=(1-|\xi^\prime|^2)^{1/2}$ be the graph function of $S^{d-1}$ near the north pole. We set $d\mu:=\psi d\sigma=\psi(\xi^\prime, \varphi(\xi^\prime))(1+|\nabla \varphi(\xi^\prime)|^2)^{1/2}d\xi^\prime.$
 In order to get estimate \eqref{sp}, we will use Stein's analytic interpolation theorem (see e.g. \cite{Lai21, St, Vg}).
 For $\mathrm{Re}(z)>0$, we define \beqq M_z(\xi):=\frac{1}{\Gamma(z)}(\xi_d-\varphi(\xi^\prime))_+^{z-1}\chi(\xi_d-\varphi(\xi^\prime))\psi(\xi^\prime, \varphi(\xi^\prime))(1+|\nabla \varphi(\xi^\prime)|^2)^{1/2},\eeqq where $(\cdot)_+$ refers to the positive part of a function, $\Gamma$ is the Gamma function and $\chi\in C_0(\mathbb{R})$ is a smooth cut-off function which equals $1$ on the unit ball $B(0,1)$ and equals $0$ outside the ball $B(0,2)$.
Then $M_z$ extends to a distribution-valued analytic function on the whole complex plane $\mathbb{C}$ by a standard argument (see e.g. page 296-300 of \cite{MS}). Let us collect some properties of $M_z$ :\\
(i). When $z=0$, $(M_zf)^{\check{}}=C_0\check{f}*\check{d\mu}$, for some constant $C_0$ and $f\in\mathcal{S}(\mathbb{R}^d)$;\\
(ii). When $\mathrm{Re}z=1$, $\|M_z\|_\infty\leq C_z^1$, where $C_z^1$ has at most exponential growth in $|\mathrm{Im} z|$;\\
(iii).  When $\mathrm{Re}z=-\frac{d-1}{2}$, $\|\check{M_z}\|_\infty\leq C_z^2$, where $C_z^2$ has at most exponential growth in $|\mathrm{Im} z|.$

Let $T_z$ be the operator from $\mathcal{S}(\mathcal{R}_\theta^d)$ to $\mathcal{S}^\prime(\mathcal{R}_\theta^d)$: for
 $U_\theta(f)\in \mathcal{S}(\mathcal{R}_\theta^d)$,
 \beq\label{Tz} (T_z(U_\theta(f)), U_\theta(g)):=(M_zf,\tilde{g}), \quad \text{ for\: all\:} g\in \mathcal{S}(\mathbb{R}^d).\eeq
 Then $T_\cdot$ is an operator-valued analytic function on the whole complex plane $\mathbb{C}$.
Note that when $\mathrm{Re}z=1$, $T_z(U_\theta(f))=U_\theta(M_zf)$, then by Lemma \ref{HY}, \beq\label{T2}\|T_z(U_\theta(f))\|_{L_2(\mathcal{R}_\theta^d)}\leq C_z^1\|U_\theta(f)\|_{L_2(\mathcal{R}_\theta^d)}.\eeq
 When $\mathrm{Re}z=-\frac{d-1}{2}$, $T_z(U_\theta(f))=(id_{\mathcal{R}_\theta^d}\otimes\tau_\theta)(id_{\mathcal{R}_\theta^d}\otimes U_\theta(f))(\pi_\theta(\check{M_z}))$. As in the proof of Lemma \ref{1-0},
 \beq\label{Tinfty}\|T_z(U_\theta(f))\|_{L_\infty(\mathcal{R}_\theta^d)}\leq C_z^2\|U_\theta(f)\|_{L_1(\mathcal{R}_\theta^d)}.\eeq
Applying Stein's analytic interpolation theorem to \eqref{T2} and \eqref{Tinfty}, one has
\beq\label{Tp} \|T_0(U_\theta(f))\|_{L_{p^\prime}(\mathcal{R}_\theta^d)}\lesssim\|U_\theta(f)\|_{L_p(\mathcal{R}_\theta^d)}.\eeq
As in the non-endpoint case, given $x=U_\theta(f)\in\mathcal{S}(\mathcal{R}_\theta^d)$, by the property (i) of $M_z$,
\begin{equation}\label{tz1}
 C_0\|f\|_{L_2(S^{d-1},\psi d\sigma)}^2 =C_0\langle \check{f}, \check{f}*\check{d\mu}\rangle_{L_2(\mathbb{R}^d)}
=\langle\check{f}, (M_0f)^{\check{}} \rangle_{L_2(\mathbb{R}^d)}
=(M_0f,\bar{f}).
\end{equation}
Here the finally equality holds in the sense of distribution. By \eqref{Tz}, \eqref{Tp}, the H\"{o}lder inequality and Lemma \ref{f*g},
\begin{align}\label{tz2}
  (M_0f,\bar{f})&=(T_0(U_\theta(f)),U_\theta(\bar{\tilde{f}}))\nonumber\\
  & =\tau_\theta(T_0(U_\theta(f))U_\theta(\bar{\tilde{f}}))\nonumber \\
  &\leq \|T_0(U_\theta(f))\|_{L_{p^\prime}(\mathcal{R}_\theta^d)}\|U_\theta(\bar{\tilde{f}})\|_{L_{p}(\mathcal{R}_\theta^d)}\nonumber\\
  &\lesssim \|U_\theta(f)\|_{L_{p}(\mathcal{R}_\theta^d)}^2.
  \end{align}
Then \eqref{sp} follows from \eqref{tz1} and \eqref{tz2}.
  \end{proof}
\noindent {\bf Acknowledgements} \
G. Hong and L. Wang were supported by National Natural Science Foundation of China (No. 12071355). X. Lai was supported by National Natural Science Foundation of China (No. 11801118), Fundamental Research Funds for the Central Universities (No. FRFCU5710050121) and China Postdoctoral Science Foundation (No. 2017M621253, No. 2018T110279).
\bibliographystyle{amsplain}

\end{document}